\documentclass[leqno,11pt]{amsart} 
\usepackage{amssymb, amsmath, color, amstext}
\usepackage[english]{babel}
\usepackage{filecontents}
\theoremstyle{plain}
\newtheorem{theo}{\indent\textbf Theorem}[section]
\newtheorem{lemma}[theo]{\indent\textbf Lemma}
\newtheorem{coro}[theo]{\indent\textbf Corollary}
\newtheorem{prop}[theo]{\indent\textbf Proposition}
\theoremstyle{definition} 
\newtheorem{defi}[theo]{\indent\textbf Definition}
\newtheorem{remark}[theo]{\indent\textbf Remark}
\newcommand{\vna}{von Neumann algebra}
\newcommand{\cohaco}{Cowling-Haagerup constant}
\newcommand{\sufa}{subfactor}

\newcommand{\ucp}{unital completely positive}
\newcommand{\cobo}{completely bounded}
\newcommand{\syenal}{symmetric enveloping algebra}
\newcommand{\syenin}{symmetric enveloping inclusion}
\newcommand{\stin}{standard invariant}
\newcommand{\biha}{Bisch-Haagerup}
\newcommand{\tII}{type II$_1$}
\newcommand{\codi}{countable discrete}

\newcommand{\NM}{{N\subset M}}
\newcommand{\tN}{\tilde N}
\newcommand{\tM}{\tilde M}
\newcommand{\tS}{\tilde S}
\newcommand{\tT}{\tilde T}
\newcommand{\TS}{{T\subset S}}

\newcommand{\op}{\text{op}}
\newcommand{\MMop}{M\overline\otimes M^\op}
\newcommand{\MboxM}{M\boxtimes_{\eN} M^\op}

\newcommand{\eN}{{e_N}}

\newcommand{\Lcb}{\Lambda_{\text{cb}}}
\newcommand{\cb}{\text{cb}}
\newcommand{\bdd}{\text{bdd}}
\newcommand{\Vcb}{\Vert_\cb}
\newcommand{\G}{\mathcal G}

\newcommand{\GNM}{\mathcal G_{N,M}}

\newcommand{\C}{\mathbb C}
\newcommand{\Z}{\mathbb Z}
\newcommand{\N}{\mathbb N}
\newcommand{\F}{\mathbb F}
\newcommand{\E}{\mathbb E}
\newcommand{\mQ}{\mathcal Q}
\newcommand{\mN}{\mathcal N}
\newcommand{\mM}{\mathcal M}

\newcommand{\spann}{\text{span}}
\newcommand{\loriar}{\longrightarrow}
\newcommand{\sumk}{\sum_{k\in K}}
\newcommand{\tvpl}{\tilde\varphi_l}
\newcommand{\supp}{\text{supp}}
\newcommand{\ootimes}{\overline\otimes}

\newcommand{\CB}{\text{CB}}
\newcommand{\A}{\mathcal A}
\newcommand{\al}{\alpha}
\newcommand{\gO}{\geqslant 0}

\begin{document}
\title{Weak amenability for subfactors}
\maketitle
\begin{center}
{\sc by Arnaud Brothier\footnote{Vanderbilt University, Department of Mathematics, 1326 Stevenson Center Nashville, TN, 37212, USA,\\ arnaud.brothier@vanderbilt.edu}
}
\end{center}

\begin{abstract}
We define the notions of weak amenability and the Cowling-Haagerup constant for extremal finite index \sufa s of \tII. 
We prove that the \cohaco\ only depends on the standard invariant of the subfactor. 
Hence, we define the \cohaco\ for \stin s.
We explicitly compute the constant for Bisch-Haagerup subfactors and prove that it is equal to the constant of the group involved in the construction.
Given a finite family of amenable \stin s, we prove that their free product in the sense of Bisch-Jones is weakly amenable with constant 1.
We show that the \cohaco\ of the tensor product of a finite family of \stin s is equal to the product of their \cohaco s.
\end{abstract}
\begin{center}
\end{center}
\section*{Introduction and main results}
Jones initiated subfactors theory and discovered the celebrated Jones Polynomial for knots \cite{Jones_index_for_subfactors,Jones_hecke_braid_group,Jones_polynome_vna}.
Subfactors theory is connected to many other areas of mathematics and physics such as topological quantum field theory \cite{Evans_Kawahigashi_book}, conformal field theory \cite{Evans_Kawahigashi_92_sf_cft}, statistical mechanics \cite{Jones_knot_stat_mecanic}, and so on.
Given a \sufa\ $\NM$ (a unital inclusion of factors of \tII\ with finite index), Jones associated a combinatorial invariant $\GNM$ called the \stin.
It has been axiomatized as a paragroup by Ocneanu \cite{Ocneanu_quant_group_string_galois}, as a $\lambda$-lattice by Popa \cite{popa_system_construction_subfactor}, and as a planar algebra by Jones \cite{jones_planar_algebra}.
One can extract a rigid C$^*$-tensor category from a \stin. 
Exotic fusion categories have been constructed using this process, see the survey \cite{Jones_Morrison_Snyder_14_Classification}.
Popa proved that any \stin\ can be obtained via a \sufa, see \cite{Popa_Markov_tr_subfactors,popa_system_construction_subfactor,Popa_univ_constr_subfactors}.
Popa and Shlyakhtenko proved in \cite{Popa_Shlyakhtenko_univ_prop_LF_subfactor} that this \sufa\ can be constructed as a \sufa\ of the free group factor with infinitely many generators $L(\mathbb F_\infty)$.
Later on, those results have been proved using planar algebras and free probability techniques in \cite{GJS_random_matrices_free_proba_planar_algebra_and_subfactor,JSW_orthogonal_approach_planar_algebra,GJS_semifinite_algebra,Hartglass_13_GJS} and independently in \cite{Kodiyalam_Sunder_11_GJS}.

Popa initiated the study of analytic properties of \sufa s \cite{Popa_94_book,Popa_classification_subfactors_amenable,Popa_symm_env_T}. 
Such properties are relevant for infinite depth \sufa s. 
The main examples of infinite depth \sufa s are the Temperley-Lieb-Jones subfactors \cite{Jones_index_for_subfactors}, the diagonal \sufa\ \cite{Popa_classification_subfactors_amenable} the \biha\ \sufa s \cite{Bisch_Haagerup_composition_subfactors}, and free product of subfactors \cite{Bisch_Jones_97_Fuss_Catalan,Bisch_Jones_Free_composition} such as the Fuss-Catalan subfactors.
Popa defined the notion of amenability and strong amenability for subfactors and proved that a strongly amenable subfactor is completely described by its \stin\ \cite{Popa_classification_subfactors_amenable}.
This generalized the Reconstruction Theorem of Ocneanu for finite depth \stin s \cite{Ocneanu_quant_group_string_galois}.
Later on, he defined the notion of property (T) for \sufa s \cite{Popa_symm_env_T}. 
Bisch and Popa provided examples of such \sufa s in \cite{Bisch_Popa_subfactors_PropT}.
In \cite{Bisch_Nicoara_Popa_cont_families_same_standard_inv}, Bisch et al. used this notion to prove that there exists uncountably many pairwise non-isomorphic hyperfinite \sufa s at index 6 with the same \stin.
In \cite{Brot_Vaes_13_subfactor_fundamental_group}, using a spectral gap argument, the author and Vaes showed that there exists families of hyperfinite \sufa s at index 6 with the same \stin\ that are non-classifiable by countable structures. 
Note that random walks on fusion algebras associated to subfactors have been studied by Bisch and Haagerup in \cite{Bisch_Haagerup_composition_subfactors} and by Hiai and Izumi in \cite{Hiai_Izumi_98_fusion_alg}.

Given an extremal \sufa\ $\NM$ of \tII, Ocneanu associated to it the asymptotic inclusion $M\vee(M'\cap M_\infty)\subset M_\infty$, where $M_\infty$ is the enveloping algebra of the Jones tower \cite{Ocneanu_quant_group_string_galois}.
This is analogous to the quantum double in the context of \sufa s \cite{Drinfeld_86_ICM}.
Longo and  Rehren gave another construction of quantum doubles for type III \sufa s \cite{Longo_Rehren_95_Nets_sf}.
Popa defined the \syenin\ of an extremal finite index \sufa\ $\NM$ of \tII\ \cite{Popa_94_Sym_env_alg}. 
If $\NM$ is a hyperfinite \sufa\ of \tII\ with finite depth, then the Popa's \syenin\ and the Ocneanu's asymptotic inclusion are isomorphic. 
The \syenin\ is particularly adapted to the study infinite depth \sufa s. 
Property (T) and amenability of a subfactor can be defined as a co-property of the \syenin\ \cite{Popa_symm_env_T}.
Moreover, those properties only depends on the \stin\ of the subfactor.
Popa's \syenin\ can be constructed with a category of bimodules in a similar way than the Longo-Rehren inclusion, see \cite{Masuda_00_LR-construction}.
Later on, Curran et al. constructed in \cite{Curran_Jones_Shlyakhtenko_14_sym_env_alg} a \syenin\ directly from a \stin\ by using planar algebras.

The notion of weak amenability for locally compact groups has been introduced by Cowling and Haagerup in \cite{Cowling-Haagerup-89-completely-bounded-mult}. We remind the definition for countable discrete groups.
A countable discrete group $G$ is weakly amenable if there exists a sequence of finitely supported maps $f_n:G\longrightarrow\C$ that converges pointwise to 1 and such that $\limsup_n\Vert f_n\Vcb < \infty$.
The \cohaco\ of $G$ is the infimum of $\limsup_n\Vert f_n\Vcb < \infty$ for such sequences and is denoted by $\Lcb(G).$
Amenable groups are weakly amenable.
Classical examples of weakly amenable groups are lattices in simple real rank one Lie groups \cite{Cowling-Haagerup-89-completely-bounded-mult}, hyperbolic groups \cite{Ozawa_08_wa_hyperbolic}, or $CAT(0)$-cubical groups \cite{Guentner_Higson_10_wa_cat0}.
Analogous properties and constants for C$^*$-algebras and \vna s have been defined by Haagerup and are called respectively by the completely bounded approximation property (CBAP) and the weak* completely bounded approximation property (W$^*$-CBAP) \cite{Haagerup_Calg_no_cbap}.
Those notions play an important role in classification of group \vna s and in Popa's deformation/rigidity theory \cite{Cowling-Haagerup-89-completely-bounded-mult,Popa_survey_Deformation_rigidity_group_actions_and_vna}.
Note, the notion of weak amenability has been defined and studied for Kac algebras and locally compact quantum groups \cite{Kraus_Ruan_99_AP_Kac_alg, Junge_Neufang_Ruan_09_Rep_th_lcqgroups, Freslon_12_free_product, Freslon_13_weak_amen,DeCommer_Freslon_Yamashita_CCAP}.

In the present article, we define and study the notion of weak amenability and \cohaco\ for extremal finite index \sufa s of \tII.
This is done by following the idea of Popa and his definition of property (T) for \sufa s.
In section \ref{sec:notations}, we introduce our setup and recall classical notations.
In section \ref{sec:CHconstant}, we prove that the \cohaco\ only depends on the \stin\ of the subfactor. 
We prove it by adapting a strategy due to Popa \cite[Section 9]{Popa_symm_env_T}.
This allows us to define the Cowling-Haagerup constant of a \stin.
In section \ref{sec:diagonal}, we compute the \cohaco\ of a diagonal subfactor, which is equal to the \cohaco\ of the group involved. 
This shows that our notion of weak amenability coincides with the classical notion for finitely generated groups.
In section \ref{sec:BH}, we show that the \cohaco\ of a Bisch-Haagerup subfactor is equal to the \cohaco\ of the group involved. We follow a strategy due to Bisch and Popa that we adapt to the context of completely bounded maps \cite{Bisch_Popa_subfactors_PropT}.
In section \ref{sec:free_product}, we consider free product of \stin s defined by Bisch and Jones \cite{Bisch_Jones_97_Fuss_Catalan,Bisch_Jones_Free_composition}.
Given a finite family of amenable \stin s, we prove that their free product is a weakly amenable \stin\ with constant 1. We use a description of the \syenin\ due to Masuda \cite{Masuda_00_LR-construction} and some results due to Ricard and Xu \cite{Ricard_Xu_06_Khintchine}.
In section \ref{sec:tensor_product}, we prove that the \cohaco\ of the tensor product of a finite family of \stin s is equal to the product of their \cohaco s. The proof is similar than the one for groups given by Cowling and Haagerup \cite{Cowling-Haagerup-89-completely-bounded-mult}.

\subsection*{Acknowledgement}
I express my gratitude to Dietmar Bisch for many useful discussions and encouragements. 
I thank Stefaan Vaes for helpful comments on an earlier version of this manuscript.
I also thank Vaughan Jones, Jesse Peterson, and Jacqueline Kirby for stimulating discussions.
\section{Notations and setup}\label{sec:notations}
In the article, any finite index \sufa\ is assumed to be of type II$_1$ and extremal. 
Any \stin s will be supposed to be extremal as well.
Let $\NM$ be a finite index \sufa.
We consider its Jones tower 
$$N\subset M\subset M_1\subset M_2\subset \cdots,$$ 
where $N=M_{-1}$ and $M=M_0$.
We write the Jones projections $e_j, j\geqslant 1$ such that $e_N=e_1$ and $M_{j+1}=\{M_j,e_{j+1}\}''$.
We denote the standard invariant of $\NM$ by $\GNM$, or simply $\G$ if the context is clear.
Let $\Gamma_{N,M}$ be the dual principal graph of $\NM$, where the even vertices $K$ are in bijection with the isomorphism classes of irreducible $M$-bimodules that appear in the Jones tower.
The symmetric enveloping inclusion associated to $\NM$ is denoted by 
$$M\vee M^\op\subset \MboxM,$$ 
or simply by $\TS$.
We fix a tunnel for $\NM$ that we denote by 
$$M_0\supset M_{-1}\supset M_{-2}\supset M_{-3}\supset\cdots,$$ 
and continue to denote by $e_j\in M_j, j\in\Z$ the Jones projections for positive or negative integers.
Following \cite[Lemma 1.4]{Popa_symm_env_T}, we fix an embedding of $\bigcup_j M_j$ and $\bigcup_j M_j^\op$ in $S$ via the choice of the tunnel.

\section{The Cowling-Haagerup constant for subfactors and standard invariants}\label{sec:CHconstant}

In the article, any inclusion of tracial \vna s will be supposed to be unital and tracial.

\begin{defi}\label{defi:dbapid}
Consider an inclusion of tracial \vna s $\mathcal N\subset (\mathcal M,\tau).$ 
A completely bounded approximation of the identity (CBAI) for $\mathcal N\subset (\mathcal M,\tau)$ with constant $C$ is a sequence of maps 
$$\varphi_l:\mathcal M\longrightarrow \mathcal M,\ l\geqslant 0$$ 
such that 
\begin{enumerate}
\item $\varphi_l$ is normal,
\item $\varphi_l$ is a $\mathcal N$-bimodular map,
\item the range of $\varphi_l$ is a bifinite $\mathcal N$-bimodule, i.e. it is finitely generated as a left $\mathcal N$-module and as a right $\mathcal N$-module,
\item $\Vert \varphi_l(x)-x\Vert_2\longrightarrow_l 0,$ for any $x\in\mathcal M$, and
\item $\limsup_l\Vert\varphi_l\Vcb\leqslant C,$ where $\Vert\cdot\Vcb$ denote the completely bounded norm.
\end{enumerate}
\end{defi}

The next definition is inspired by Popa's definition of property (T) for \sufa s \cite{Popa_symm_env_T}.

\begin{defi}\label{defi:Lcb}
Consider a \sufa\ $\NM$ and its associated \syenin\ $\TS$.
The \sufa\ $\NM$ is weakly amenable if there exists a CBAI with a finite constant $C$ for $\TS$.
The \cohaco\ of $\NM$ is the lower bound of the constants $C$ such that $\TS$ admits a CBAI with constant $C$.
We denote it by 
$$\Lcb(N,M).$$
It $\NM$ is not weakly amenable we write $\Lcb(N,M)=\infty.$
\end{defi}

\begin{remark}\label{rem:amenable}
Popa introduced the notion of amenability for \sufa s and \stin s in \cite{Popa_94_book,Popa_classification_subfactors_amenable,Popa_symm_env_T}.
A \sufa\ $\NM$ is amenable if $N$ and $M$ are the hyperfinite II$_1$ factor (the \sufa\ is hyperfinite) and the index is equal to the square of the norm of the principal graph, i.e. $[M:N]=\Vert\Gamma_{N,M}\Vert^2.$
Popa showed that the later condition is equivalent to the following approximation property, see \cite[Remark 3.5.4]{Popa_betti_numbers_invariants}.
Let $\TS$ be the \syenin\ of $\NM$.
There exists a sequence of maps
$$\varphi_l:S\loriar\ S,\ l\geqslant 0$$ 
such that 
\begin{enumerate}
\item $\varphi_l$ is normal,
\item $\varphi_l$ is a $T$-bimodular map,
\item $\varphi_l(\mathcal S)$ is a bifinite $T$-bimodule,
\item there exists a constant $A$ such that $\tau\circ\varphi_l\leqslant A\tau$, where $\tau$ is the unique normal faithful trace on $S$,
\item $\Vert \varphi_l(x)-x\Vert_2\longrightarrow_l 0,$ for any $x\in\mathcal M$, and
\item $\varphi_l$ is unital and completely positive.
\end{enumerate}
In particular, an amenable \sufa\ is weakly amenable with constant 1.
\end{remark}

In the rest of this section, we prove that $\Lcb(N,M)$ only depends on the \stin\ of $\NM$.
We adapt a proof of Popa for completely bounded maps \cite[Section 9]{Popa_symm_env_T}.
The next proposition recall some properties due to Popa of the \syenin.

\begin{prop}\cite[Theorem 4.5]{Popa_symm_env_T}\label{prop:bimod_decompo}
Let $\NM$ be a subfactor and let $\TS$ be its associated \syenin.
Then $\TS$ is an irreducible \sufa\ of \tII.

Denote by $\{H_k,\ k\in K\}$ the set of irreducible $M$-bimodules indexed by the even vertices of the dual graph of $\NM$.
Let $K_n$ be the vertices at distance smaller than $2n$ from the root of the graph.
Then for any $k\in K_n$, there exists $v_k\in M_{-n}'\cap M_n\cap S$ such that $L^2(\spann Tv_kT)$ is isomorphic to the $T$-bimodule $H_k\otimes \overline H_k^\op$.
Further, we have the following decomposition of $T$-bimodules 
$$L^2(S)=\sumk L^2(\spann Tv_kT)\simeq_{T-T}\bigoplus_{k\in K} H_k\otimes \overline H_k^\op .$$
The spaces $L^2(\spann Tv_kT)$ are irreducible $T$-bimodules and are pairwise non-isomorphic.
We write $s_k$ the orthogonal projection from $L^2(S)$ onto $L^2(\spann Tv_kT)$.
\end{prop}

The next proposition gives a useful characterization of $T$-bimodular maps.

\begin{prop}\label{prop:scalar_functions}
Suppose the same assumptions as in Proposition \ref{prop:bimod_decompo}.
Let $a_k:=s_k\vert_S$ be the restriction of the orthogonal projection $s_k$ to $S$.
\begin{enumerate}
\item Then, the range of $a_k$ is contained in $S$ and is equal to $S\cap L^2(\spann Tv_kT).$
\item Let $\varphi:S\loriar S$ be  a $T$-bimodular map such that $\varphi(S)$ is a bifinite $T$-bimodule
Then $\varphi$ is normal and belongs to the vector space generated by the set $\{a_k,\ k\in K\}$.
We write $c_\varphi:K\loriar \C$ the unique scalar valued function that satisfies 
$$\varphi=\sumk c_\varphi(k)a_k.$$
\item The \syenin\ $\TS$ admits a CBAI with constant $C$ if and only if there exists a sequence of finitely supported scalar valued functions $(c_l,\ l\geqslant 0)$ such that 
\begin{enumerate}
\item $c_l(k)\loriar_l 1$, for any $k\in K$, and
\item $\limsup_l\Vert \sumk c_l(k)a_k\Vcb=C$.
\end{enumerate}
\end{enumerate}
\end{prop}

\begin{proof}
Let $\varphi:S\loriar S$ be a $T$-bimodular map such that its image is a bifinite $T$-bimodule.
Consider the algebraic $T$-bimodule $\mathcal V_k:=L^2(spTv_kT)\cap S,$ where $k\in K$.
The $\mathcal V_k$ are pairwise non-isomorphic irreducible $T$-bimodules.
Hence the space of intertwiners between $\mathcal V_k$ and $\mathcal V_l$ is trivial if $k\neq l$ and is isomorphic to $\C$ if $k=l$.
By \cite[Proof of Lemma 5.17]{Falguiere_these}, we have that $s_k(S)=\mathcal V_k$.
Therefore, $s_k\circ\varphi\circ s_l\vert_S$ is a $T$-bimodular map between $\mathcal V_l$ and $\mathcal V_k$.
Hence, this map is identically equal to zero if $k\neq l$.
If $k=l$ there exists a scalar $c_\varphi(k)\in\C$ such that $s_k\circ\varphi\circ s_k\vert_S=c_\varphi(k)s_k\vert_S$.
Let $\supp(\varphi)$ be the set of $k\in K$ such that $c_\varphi(k)\neq 0$.
By assumption, the image of $\varphi$ is a bifinite $T$-bimodule.
Hence, $\supp(\varphi)$ is finite and we get 
$$\varphi=\sum_{k\in\supp(\varphi)}c_\varphi(k)s_k\vert_S.$$
The rest of the proposition follows.
\end{proof}

\begin{remark}\label{rem:amenable2}
Proposition \ref{prop:scalar_functions} implies that a \sufa\ $\NM$ is amenable if and only it is a hyperfinite \sufa\ and there exists a sequence of finitely supported scalar valued functions $(c_l,\ l\geqslant 0)$ such that 
\begin{enumerate}
\item $c_l(k)\loriar_l 1$, for any $k\in K$, and
\item $\sumk c_l(k)a_k:S\loriar S$ is a completely positive map,
\end{enumerate} 
where $\TS$ is the \syenin\ associated to $\NM$ and $K$ is the set vertices of the dual principal graph of $\NM$.
By considering $c_l(k_0)^{-1}c_l$, where $k_0$ is the index of the trivial $T$-bimodule $L^2(T)$, we can always assume that an amenable \sufa\ admits a sequence of \ucp\ maps that preserve the unique normal trace-preserving conditional expectation $\E_T: S\loriar T$, and satisfies the assumptions of Remark \ref{rem:amenable2}.
\end{remark}

\begin{theo}\label{theo:wa_for_stin}
Consider two finite index subfactors $\NM$ and $\tilde N\subset\tilde M$ with isomorphic \stin s.
Then $$\Lcb(N,M)=\Lcb(\tilde N,\tilde M).$$
\end{theo}

\begin{proof}
\textbf{Step 1.}\textit{
We assume that we have a non-degenerate commuting square
$$\begin{array}{ccc}
\tilde N&\subset&\tilde M\\
\cup&&\cup\\
N&\subset & M
\end{array},$$
such that $N'\cap M_n=\tilde N'\cap \tilde M_n$ for any $n\geqslant 0$.}

Let us show that $\Lcb(\tN,\tM)\leqslant \Lcb(N,M).$
Let $\{M_j,\ j\leqslant 0\}$ be a tunnel for $\NM$.
Denote by $\{e_j,\ j\in\Z\}$ the set of Jones projection associated to the tunnel-tower of $\NM$.
For any $j\leqslant -2,$ consider the \vna\ $\tilde M_j=\tilde M\cap \{e_{j+2}\}'$.
This defines a tunnel for $\tilde N\subset\tilde M$ such that $\tilde M_i'\cap \tilde M_j=M_i'\cap M_j$ for any $i,j\in\Z$.
Denote by $\TS$ and $\tilde T\subset \tilde S$ the \syenin s associated to $\NM$ and $\tilde N\subset\tilde M$ respectively.
Those tunnels define inclusions of $M_j$ in $S$ and $\tilde M_j$ in $\tilde S$ for any $j$ such that we have a commuting square 
$$\begin{array}{ccc}
\tilde M_j&\subset&\tilde S\\
\cup&&\cup\\
M_j&\subset & S
\end{array}.$$
In particular, we have equality of the relative commutants $\tilde M_i'\cap \tilde M_j=M_i'\cap M_j$ for any $i,j\in\Z$ as subalgebras of $\tS.$
Consider a CBAI $(\varphi_l,\ l\geqslant 0)$ for $\TS$ with constant $C=\Lcb(N,M)$.
We fix a natural number $l$.
We want to construct a normal $\tilde T$-bimodular map $\tilde\varphi_l:\tilde S\loriar\tilde S$ such that $\Vert \tilde \varphi_l\Vcb=\Vert\varphi_l\Vcb$ and $\tilde\varphi_l\vert_S=\varphi_l$.
We drop the subscript $l$ and consider a single map $\varphi:=\varphi_l$.
Consider the commuting square 
\begin{equation}\label{equa:square1}\begin{array}{ccc}
\tilde T&\subset&\tilde S\\
\cup&&\cup\\
T&\subset & S
\end{array},\end{equation}
the Jones projection $e:L^2(\tS)\longrightarrow L^2(S)$, and the basic construction $\langle \tS,e\rangle$.
We denote by $Tr$ the trace, possibly semi-finite, of the basic construction $\langle \tS,e\rangle.$

Let us show that the commuting square (\ref{equa:square1}) is non-degenerate.
Proposition \ref{prop:bimod_decompo} implies that the \vna\ $\tS$ is generated by $\tT$ and the family of relatives commutants $\tilde M_{-n} ' \cap M_n$ for $n\gO$.
Note, by construction, for any $i,j\in\Z$ we have equality of the relative commutants $\tilde M_i'\cap \tilde M_j=M_i'\cap M_j$ as subalgebras of $\tS$.
Therefore, $\tS$ is generated by $\tT$ and the family of relatives commutants $M_{-n} ' \cap M_n,\ n\gO$ which are contained in $S$.
Therefore, the commuting square (\ref{equa:square1}) is non-degenerate.

We fix an orthonormal basis $\{m_j,\ j\in J\}\subset \tT$ of $\tT$ over $T$.
This means that for any $x\in\tT$ we have $x=\sum_{j\in J} \E_T(xm_j)m_j^*$ and for any $i,j\in J$ we have $\E_T(m_i^*m_j)\delta_{i,j}f_i$, where $f_i$ is a projection, $\delta_{i,j}$ is the Kronecker symbol and $\E_T$ is the unique normal trace-preserving conditional expectation from $\tT$ onto $T$.
Note that $\{m_j,\ j\in J\}\subset \tT$ is an orthonormal basis of $\tS$ over $S$.
Indeed, since the square (\ref{equa:square1}) is commuting we automatically obtain the first assumption in the definition of an orthonormal basis. Since the square (\ref{equa:square1}) is commuting and non-degenerate we also fulfill the second assumption of the definition of an orthonormal basis.
Hence, for any $x\in\langle \tS,e\rangle$ there exists a unique collection $\{x_{ij}\in S,\ i,j\in J\}$ such that 
$$x=\sum_{i,j\in J}m_iex_{ij}m_j^*.$$
Consider the map 
$$\phi:\langle \tS,e\rangle\longrightarrow\langle \tS,e\rangle,\ x=\sum_{i,j\in J}m_iex_{ij}m_j^*\longmapsto \sum_{i,j\in J}m_ie\varphi(x_{ij})m_j^*.$$
This is a normal $\langle \tT,e\rangle$-bimodular map.

We claim that $\phi$ extends the map $\varphi$.
By Proposition \ref{prop:scalar_functions}, there exists a map $c:K\loriar\C$ such that $\varphi=\sumk c(k)a_k$.
The maps $\phi$ and $\varphi$ are normal and $T$-bimodular.
Hence, it is sufficient to prove that $\phi(v_k)=\varphi(v_k)=c(k)v_k$ for any $k\in K$.
Let us fix $k\in K$.
There exists $n\geqslant 0$ such that $k\in K_n$, where $K_n$ is the set of even vertices of the dual principal graph at distance smaller or equal to $2n$ from the root of the graph.
Let $\{m_j^k,\ j\in J_k\}$ be an orthonormal basis of $\tilde M_{-n}\vee \tilde M_{-n}^\op$ over $M_{-n}\vee M_{-n}^\op,$ where all such \vna s are viewed inside $\tilde T$.
Consider the commuting square
\begin{equation}\label{equa:square2}\begin{array}{ccc}
\tilde M_{-n}\vee \tilde M_{-n}^\op & \subset & \tT \\
\cup & & \cup  \\
M_{-n}\vee M_{-n}^\op & \subset & T 
\end{array}.\end{equation}
Let us show that (\ref{equa:square2}) is non-degenerate.
By construction, the commuting square 
$$\begin{array}{ccc}
\tilde M_{-n} & \subset & \tilde M \\
\cup & & \cup  \\
M_{-n} & \subset & M 
\end{array}$$
is non-degenerate.
Therefore, the following commuting square 
$$\begin{array}{ccc}
\tilde M_{-n}\ootimes \tilde M_{-n}^\op & \subset & \tilde M\ootimes\tilde M^\op \\
\cup & & \cup  \\
M_{-n}\ootimes M_{-n}^\op & \subset & M\ootimes M^\op
\end{array}$$
is non-degenerate as well.
But this commuting square is isomorphic to (\ref{equa:square2}). So we are done.
Since (\ref{equa:square1}) and (\ref{equa:square2}) are non-degenerate commuting squares we obtain that $\{m_j^k,\ j\in J_k\}$ is an orthonormal basis of $\tS$ over $S$.
The vector $v_k$ belongs to $M_{-n}'\cap M_n.$
By \cite[Proposition 2.6.b]{Popa_symm_env_T}, we have the equality 
$$(\tilde M_{-n}\vee \tilde M_{-n}^\op)'\cap\tilde S=\tilde M_n^\op\cap \tilde M_{n}.$$
Therefore, $v_k$ commutes with the whole basis $\{m_j^k,\ j\in J_k\}$.
Following the proof of \cite[Lemma 9.2]{Popa_symm_env_T}, we obtain that $\phi$ does not depend on the choice of the orthonormal basis $\{m_j,\ j\in J\}\subset \tT$ of $\tT$ over $T$.
Denote by $\E_S$ (resp. $\E_{\tT}$) the normal trace-preserving conditional expectation from $\tS$ onto $S$ (resp. onto $\tT$).
We have, 
\begin{align*}
v_k & =\sum_{i,j\in J_k} m_i^k \E_S( {m_i^k}^* v_k m_j^k) e {m_j^k}^*,\\
      & = \sum_{i,j\in J_k} m_i^k v_k \E_S( {m_i^k}^*  m_j^k) e {m_j^k}^* \text{ because }[v_k,m_i^k]=0,\\
      &= \sum_{i,j\in J_k} m_i^k v_k \E_T( {m_i^k}^*  m_j^k) e {m_j^k}^* \text{ because }\E_T\circ \E_{\tT}=\E_T.
\end{align*}

Therefore,
\begin{align*}
\phi(v_k) & =\phi(\sum_{i,j\in J_k} m_i^k v_k \E_T( {m_i^k}^*  m_j^k) e {m_j^k}^*),\\
      & = \sum_{i,j\in J_k} m_i^k \varphi(v_k \E_T( {m_i^k}^*  m_j^k) ) e {m_j^k}^*\\
      & = \sum_{i,j\in J_k} m_i^k c(k) v_k \E_T( {m_i^k}^*  m_j^k) e {m_j^k}^*\\
      &= c(k)v_k=\varphi(v_k).
\end{align*}
This proves the claim.

Let us show that $\Vert \phi\Vcb=\Vert \varphi\Vcb.$
By the Stinespring Dilation Theorem \cite[Theorem B7, p347]{Brown_Ozawa_book}, there exists a normal representation $L_\varphi:S\longrightarrow B(H_\varphi)$ and two continuous linear maps $v_\varphi,w_\varphi:L^2(S)\longrightarrow H_\varphi$ such that for any $y\in S$ 
$$\varphi(y)=v_\varphi^*L_\varphi(y)w_\varphi$$ and 
$$\Vert v_\varphi\Vert=\Vert w_\varphi \Vert=\sqrt{\Vert\varphi\Vcb}.$$
Consider the Hilbert space $\tilde H_\varphi:=\ell^2(J)\otimes\ell^2(J)\otimes H_\varphi$ and the densely defined continuous linear maps 
$$\tilde v_\varphi:L^2(\langle \tS,e\rangle,Tr)\longrightarrow \tilde H_\varphi,\ \widehat{\sum_{i,j\in J}m_iex_{ij}m_j^*} \longmapsto \sum_{i,j\in J}\delta_i\otimes \delta_j\otimes v_\varphi(\widehat{x_{ij}})$$
and
$$\tilde w_\varphi:L^2(\langle \tS,e\rangle,Tr)\longrightarrow \tilde H_\varphi,\ \widehat{\sum_{i,j\in J}m_iex_{ij}m_j^*} \longmapsto \sum_{i,j\in J}\delta_i\otimes \delta_j\otimes w_\varphi(\widehat{x_{ij}}),$$
where $\widehat z$ denote the image of an element of $\langle \tS,e\rangle$ with finite $L^2$-norm in $L^2(\langle \tS,e\rangle, Tr)$.
Consider the following normal representation $$\tilde L_\varphi:\langle \tS,e\rangle\loriar B(\tilde H_\varphi),\ x\longmapsto \sum_{i,j\in J}e_{ij}\otimes 1\otimes L_\varphi(x_{ij}),$$
where $e_{ij}$ is the partial isometry of $\ell^2(J)$ that sends $\delta_j$ to $\delta_i$ and $x=\sum_{i,j\in J} m_i e x_{ij} m_j^*\in \langle \tS,e\rangle.$
One can check that $\phi(x)=\tilde v_\varphi^* \tilde L_\varphi(x) \tilde w_\varphi,$ for any $x\in \langle \tS,e\rangle.$
Hence, $\phi$ is completely bounded and $$\Vert\phi\Vcb\leqslant \Vert \tilde w_\varphi\Vert.\Vert \tilde v_\varphi\Vert=\Vert  w_\varphi\Vert.\Vert  v_\varphi\Vert=\Vert\varphi\Vcb.$$
The maps $\phi$ and $\varphi$ coincides on $S$.
Hence, we have $\Vert \phi\Vcb=\Vert\varphi\Vcb.$

Recall that $\varphi=\sumk c(k)a_k$.
Observe, the restriction $\tilde\varphi=\phi\vert_{\tS}$ of $\phi$ to $\tS$ is equal to $\sumk c(k)\tilde a_k$,
where $\tilde a_k$ is the identity map of $L^2(\spann \tT v_k\tT)\cap \tS$.
In particular, $\tilde\varphi(\tS)$ is contained in $\tS$.
Therefore, $\tilde\varphi:\tS\loriar\tS$ is a completely bounded $\tT$-bimodular map, the $\tT$-bimodule $\varphi(\tS)$ is bifinite, and $\Vert\tilde\varphi\Vcb=\Vert \varphi\Vcb$.
Hence, if $(\varphi_l,\ l\geqslant 0)$ is a CBAI with constant $C$ for $\TS$, then $(\tilde\varphi_l,\ l\geqslant 0)$ is a CBAI with constant $C$ for $\tT\subset\tS$.
We obtain, $\Lcb(\tN,\tM)\leqslant \Lcb(N,M).$

Consider a CBAI $(\phi_l)_l$ for $\tilde T\subset \tilde S$ such that $\limsup_l\Vert\phi_l\Vcb=\Lcb(\tilde N,\tilde M)$.
We define $\varphi_l:=\E_S\circ \phi_l\vert_S,$ where $\E_S:\tilde S\loriar S$ is the unique normal trace-preserving conditional expectation.
We have $\Vert \varphi_l\Vcb\leqslant\Vert\phi_l\Vcb$ for any $l\geqslant 0$.
The maps $\varphi_l$ and $\phi_l$ have the same scalar valued function $c_l:K\loriar\C$ because they coincides on the vector space generated by the set  $\{v_k,\ k\in K\}$.
Therefore by Proposition \ref{prop:scalar_functions}, the sequence $(\varphi_l)_l$ defines a CBAI for $\TS$ such that $\limsup_l\Vert\varphi_l\Vcb\leqslant \Lcb(\tilde N,\tilde M)$.
We get $\Lcb(N,M)=\Lcb(\tilde N,\tilde M)$.

\textbf{Step 2.}
\textit{
We prove the general case.}
We follow the proof of \cite[Theorem 9.5]{Popa_symm_env_T}.
Let $\NM$ and $\tN\subset\tM$ be finite index \sufa s with isomorphic \stin s.
Let $\omega\in\beta\N-\N$ be a non-principal ultrafilter and let $N^\omega$ and $M^\omega$ be the ultrapowers of $N$ and $M$ respectively.
We have a non-degenerate commuting square 
$$\begin{array}{ccc}
N^\omega&\subset&M^\omega\\
\cup&&\cup\\
N&\subset&M
\end{array}.$$
For any $n\geqslant0$ we have $N'\cap M_n=(N^\omega)'\cap(M^\omega)_n$ by \cite[Proposition 1.10]{Pimsner_Popa_86_entropy}.
Hence, by Step 1, we have $\Lcb(N,M)=\Lcb(N^\omega,M^\omega)$.
Via the Reconstruction Theorem of Popa \cite{Popa_Markov_tr_subfactors,popa_system_construction_subfactor,Popa_univ_constr_subfactors}, we associate to $\mathcal G$ a subfactor $N^\G(R)\subset M^\G(R)$ which is constructed with the hyperfinite II$_1$ factor $R$ and the \stin\ $\G$.
By \cite{Popa_95_free_indep}, there exists an embedding of $M^\G(R)$ in $M^\omega$ which defines a non-degenerate commuting square
$$\begin{array}{ccc}
N^\omega&\subset&M^\omega\\
\cup&&\cup\\
N^\G(R)&\subset&M^\G(R)
\end{array}$$
and such that for any $n\geqslant 0$ we have $(N^\G(R))'\cap (M^\G(R))_n=(N^\omega)'\cap (M^\omega)_n.$
Hence, by Step 1, $\Lcb(N^\omega,M^\omega)=\Lcb(N^\G(R),M^\G(R)).$
Therefore, $\Lcb(N,M)=\Lcb(N^\G(R),M^\G(R)).$
Similarly, we have $\Lcb(\tilde N,\tilde M)=\Lcb(N^\G(R),M^\G(R)).$
In conclusion, $\Lcb(N,M)=\Lcb(\tilde N,\tilde M)$.
\end{proof}

\begin{defi}
We say that a \stin\ $\G$ is weakly amenable if there exists a \sufa\ $\NM$ with \stin\ $\G$ which is weakly amenable.
The \cohaco\ of a weakly amenable \stin\ $\G$ is the \cohaco\ of a \sufa\ $\NM$ with \stin\ $\G$.
We denote $\Lcb(\G)$ as this constant.
We write $\Lcb(\G)=\infty$ if $\G$ is not weakly amenable.
\end{defi}

\section{Diagonal subfactors}\label{sec:diagonal}

Consider a II$_1$ factor $P$ and an injective group morphism $\sigma:G\loriar Out(P)$ of a finitely generated group $G=\langle g_1,\cdots,g_m\rangle$ in the outer automorphism group of $P$.
Let $\tilde\sigma:G\loriar Aut(P)$ be a section of $\sigma$ in value in the automorphism group of $P$.
Let $\{e_{ij},\ i,j=0,\cdots, m\}$ be the canonical system of matrix units of the type I$_{m+1}$ factor $\mathcal M_{m+1}(\C)$.
The diagonal subfactor associated to $\tilde\sigma$ and the generating set $\{g_1,\cdots,g_m\}$ is 
$$\{\sum_{i=0}^m \tilde\sigma(g_i)(x)\otimes e_{ii},\ x\in P\}\subset P\otimes \mathcal M_{m+1}(\C),$$
where $\tilde\sigma(g_0)=1$.
It is an extremal \sufa\ of \tII\ with index equal to $(m+1)^2$.
We continue to denote by $\tilde\sigma$ the map from $G$ to the automorphism group of $P\otimes \mathcal M_{m+1}(\C)$ defined by 
$\tilde\sigma(g)(x\otimes e_{ij})=\tilde\sigma(g)(x)\otimes e_{ij}$ for $g\in G$, $x\in P$ and $0\leqslant i,j\leqslant m$.
Note that the group morphism $\sigma\otimes\sigma^\op$ has vanishing $H^3(G,\mathbb T)$ cohomology obstruction by \cite{Jones_80_action_hyperfinite}.
Then, it defines an outer action, possibly twisted by a cocycle, of $G$ on $\MMop$ where $M=P\otimes\mathcal M_{m+1}(\C).$

\begin{theo}\cite[Theorem 3.3]{Popa_symm_env_T}\label{theo:diagonal}
Let $\NM$ be the subfactor described above and let $\TS$ be its \syenin.
Then $\TS$ is isomorphic to $$\MMop\subset \MMop\rtimes_{\sigma\otimes\sigma^\op}G.$$
\end{theo}

\begin{coro}\label{coro:diagonal}
Let $\NM$ be the subfactor as above.
We have $\Lcb(N,M)=\Lcb(G)$.
\end{coro}

\begin{proof}
Let $(f_l)_l$ be a sequence of finitely supported maps that converge pointwise to 1 and such that $\limsup_l \Vert f_l\Vcb=\Lcb(G).$
Consider the multiplier 
$$\varphi_l:=1\otimes m_{f_l}:S\loriar S$$ 
defined as $\varphi_l(tu_g)=f_l(g)tu_g$, where $t\in\MMop$ and $g\in G$.
Then $(\varphi_l)_l$ is clearly a CBAI for $\MMop\subset \MMop\rtimes_{\sigma\otimes\sigma^\op} G$ such that $\limsup_l\Vert \varphi_l\Vcb=\Lcb(G).$
Hence, $\Lcb(N,M)\leqslant \Lcb(G)$.
Let $(\varphi_l)_l$ be a CBAI with constant $\Lcb(N,M)$ for $\MMop\subset \MMop\rtimes_{\sigma\otimes\sigma^\op} G.$ 
Consider the map $f_l(g):=\tau(\varphi_l(u_g)u_g^*)$, and follow the strategy of \cite[Theorem 12.3.10]{Brown_Ozawa_book}.
It follows that $\Lcb(N,M)\geqslant \Lcb(G).$
We reach the conclusion of the corollary.
\end{proof}

\begin{remark}
Consider a hyperfinite subfactor $\NM$ and its associated \syenin\ $\TS$.
Let $\Lcb(S)$ be the classical \cohaco\ of the \vna\ $S$, see \cite[page 365] {Brown_Ozawa_book}.
It is important to notice that in general $\Lcb(S)$ is different from $\Lcb(N,M)$ even if $T$ is hyperfinite.
We present an example of a hyperfinite subfactor $\NM$ such that $\Lcb(N,M)=1$ and $\Lcb(S)=\infty$.

Consider the group \vna\ $R:=L(\Z_2^{\F_2\times \Z}\rtimes \Z)$, where $\F_2$ is the free group with two generators.
By the Fundamental Theorem of Connes \cite{Connes_76_classification_inj_factors}, $L(\Z_2^{\F_2\times \Z}\rtimes \Z)$ is isomorphic to the hyperfinite II$_1$ factor because the group $\Z_2^{\F_2\times \Z}\rtimes \Z$ is amenable and has infinite conjugacy classes.
Consider the Bernoulli action of $\F_2$ on $\Z_2^{\F_2\times \Z}$.
This induces an outer action $\theta$ of $\F_2$ on $R$.
Let $\NM$ be a diagonal subfactor associated to the action $\theta$ and any finite generating set of the group $\F_2$.
By Theorem \ref{theo:diagonal}, we have that the \syenin\ $\TS$ associated to $\NM$ is isomorphic to 
$$M\ootimes M^\op\subset M\ootimes M^\op\rtimes_{\theta\otimes\theta^\op}\F_2.$$
Therefore, $\Lcb(N,M)=\Lcb(\F_2)$ by Corollary \ref{coro:diagonal}.
Hence, $\Lcb(N,M)=1$ by \cite{Haagerup-79-An-ex-of-Non-nucl,Canniere_Haagerup_85_weak_amenability, Boszejko_Picardello_93_weak_amenable_groups}.
The \syenal\ $S$ is isomorphic to the group \vna\ 
$$L([(\Z_2^{\F_2\times \Z}\rtimes \Z) \times (\Z_2^{\F_2\times \Z}\rtimes \Z) ]\rtimes \F_2).$$
Hence, $L(\Z_2^{\F_2}\rtimes \F_2)$ is a von Neumann subalgebra of $S$.
But the group $\Z_2^{\F_2}\rtimes \F_2$ is not weakly amenable by \cite{Ozawa_Popa_factor_un_cartan,Ozawa_12_not_wa}.
Recall the \cohaco\ of a \codi\ group and the \cohaco\ of its associated group \vna\ coincides.
Further, if $H<G$ is an inclusion of countable discrete groups, then $\Lcb(H)\leqslant \Lcb(G).$
Therefore, 
$$\Lcb(S)=\Lcb( [ ( \Z_2^{\F_2\times \Z}\rtimes \Z) \times (\Z_2^{\F_2\times \Z}\rtimes \Z) ]\rtimes \F_2) \geqslant \Lcb(\Z_2^{\F_2}\rtimes \F_2)= \infty\neq 1=\Lcb(N,M).$$
\end{remark}

\section{Bisch-Haagerup subfactors}\label{sec:BH}

We follow a similar strategy developed by Bisch and Popa in \cite{Bisch_Popa_subfactors_PropT}. 

\begin{prop}\label{prop:CBAI_finite_index}
Consider a chain of II$_1$ factors $\mQ\subset \mN\subset \mM$ such that $I:=[\mN:\mQ]<\infty$.
Let $p\in\mN$ be a non-zero projection.
Then the following assertions are equivalent:
\begin{enumerate}
\item $\mN\subset \mM$ admits a CBAI with constant $C$.
\item $\mQ\subset \mM$ admits a CBAI with constant $C$.
\item $p\mN p\subset p\mM p$ admits a CBAI with constant $C$.
\end{enumerate}
\end{prop}

\begin{proof}
Proof of $(1)\Rightarrow (2)$.
Let $(\varphi_l)_l$ be a CBAI with constant $C$ for $\mN\subset\mM$.
By assumption $\varphi_l$ is a $\mN$-bimodular map and so is a $\mQ$-bimodular map.
The range of $\varphi_l$ is a bifinite $\mN$-bimodule and $[\mN:\mQ]$ is finite.
Hence, the range of $\varphi_l$ is a bifinite $\mQ$-bimodule.
Therefore, $(\varphi_l)_l$ defines a CBAI with constant $C$ for $\mQ\subset\mM$.

Proof of $(2)\Rightarrow (1)$.
Let $(\varphi_l)_l$ be a CBAI with constant $C$ for $\mN\subset\mM$.
Consider a finite orthonormal basis $\{m_j,\ j\in J\}\subset \mN$ of $\mN$ over $\mQ$,
i.e. $\sum_{j\in J}m_jE_\mQ(m_j^*x)=x,$ for any $x\in \mN$.
Consider the map 
$$\tvpl:\mM\loriar \mM,\ x\longmapsto \frac{1}{I^2}\sum_{i,j\in J}m_i\varphi_l(m_i^* x m_j)m_j^*.$$ 
If $x\in\mM$ and $y\in\mN$, then 
\begin{align*}
\tvpl(yx)&=\frac{1}{I^2}\sum_{i,j\in J}m_i\varphi_l(m_i^* yx m_j)m_j^*\\
&=\frac{1}{I^2}\sum_{i,j,k\in J}m_i\varphi_l(E_Q(m_i^* y m_k) m_k^*x m_j)m_j^*\\
&=\frac{1}{I^2}\sum_{i,j,k\in J}m_i E_Q (m_i^* y m_k ) \varphi_l(m_k^* x m_j)m_j^*\\
&=\frac{1}{I^2}\sum_{i,k\in J} y m_k \varphi_l( m_k^*  x m_j ) m_j^*=y \tvpl (x ). 
\end{align*}
Hence, $\tvpl$ is a left $\mN$-modular map and a similar computation on the right shows that $\tvpl$ is a $\mN$-bimodular map.

Let us show that $\Vert\tilde\varphi_l\Vcb\leqslant \Vert\varphi_l\Vcb$ for a fixed $l\geqslant 0$.
By the Stinespring Dilation Theorem there exists a normal representation $\pi:\mM\loriar B(H)$ and two continuous linear maps $v,w:L^2(\mM)\loriar H$ such that for any $x\in \mM,$ $\varphi_l(x)=v^*\pi(x)w$ and $\Vert v\Vert=\Vert  w\Vert =\sqrt{\Vert \varphi_l\Vcb}$.
Consider the maps 
$$\tilde v:=\frac{1}{I}\sum_{j\in J} \pi(m_j) v m_j^*\text{ and }\tilde w:=\frac{1}{I}\sum_{j\in J} \pi(m_j) w m_j^*.$$
We observe that $\tvpl(x)=\tilde v^*\pi(x)\tilde w$ for any $x\in\mM$.
Therefore, $\Vert \tvpl\Vcb\leqslant \Vert \tilde v\Vert \Vert \tilde w \Vert.$
Let us show that $\Vert \tilde v \Vert \leqslant \Vert v\Vert$.
Consider the operator 
$$v_2:=\left( \begin{array}{cc} 0 & v\\ v^* & 0 \end{array} \right) \in B(H\oplus L^2(\mM)).$$
This is a self-adjoint operator.
So its norm is equal to its spectral radius and we have 
$$-\Vert v_2\Vert 1\leqslant v_2 \leqslant \Vert v_2 \Vert 1.$$
Note that $\Vert v_2 \Vert = \Vert v \Vert.$
Consider the operator
$$\tilde v_2:=\frac{1}{I}\sum_{j\in J} \left( \begin{array}{cc} \pi(m_j) & 0\\ 0 & m_j \end{array} \right) v_2 \left( \begin{array}{cc} \pi(m_j) & 0\\ 0 & m_j \end{array} \right)^*= \left( \begin{array}{cc} 0 & \tilde v\\ \tilde v^* & 0 \end{array} \right).$$
We have 
 $$-\Vert v\Vert \frac{1}{I}\sum_{j\in J} m_jm_j^*  \leqslant \tilde v_2 \leqslant \Vert v \Vert \frac{1}{I}\sum_{j\in J} m_jm_j^*.$$
But, $\frac{1}{I}\sum_{j\in J} m_jm_j^*=1$.
So $\Vert \tilde v_2\Vert\leqslant \Vert v\Vert$.
But $\Vert \tilde v_2\Vert=\Vert \tilde v\Vert$.
Hence, $\Vert \tilde v\Vert\leqslant \Vert v\Vert$.
We get also that $\Vert \tilde w\Vert\leqslant \Vert w\Vert$. 
Therefore, $\Vert \tvpl\Vcb\leqslant \Vert v\Vert\Vert w\Vert=\Vert\varphi_l\Vcb.$

Let $x\in \mM$.
We have 
\begin{align*}
\Vert \tvpl(x)-x\Vert_2 & = \Vert \frac{1}{I^2}\sum_{i,j\in J}m_i\varphi_l(m_i^* x m_j)m_j^*-x\Vert_2\\
& = \Vert \frac{1}{I^2}\sum_{i,j\in J}(m_i\varphi_l(m_i^* x m_j)m_j^* - m_i m_i^* x m_j m_j^*) \Vert_2\\
& \leqslant  \frac{1}{I^2}\sum_{i,j\in J} \Vert m_i\Vert\Vert m_j \Vert\Vert \varphi_l(m_i^* x m_j) - m_i^* x m_j \Vert_2\loriar_l 0.\\
\end{align*}

Because $J$ is finite and the range of $\varphi_l$ is a bifinite $\mQ$-bimodule we have that the range of $\tvpl$ is a bifinite $\mN$-bimodule.
Therefore, $(\tvpl)_l$ defines a CBAI with constant $C$ for $\mN\subset \mM$.

Proof of $(1)\Rightarrow (3)$.
Consider $(\varphi_l)_l$ a CBAI with constant $C$ for $\mN\subset\mM$.
Define $(\phi_l)_l$, where $\phi_l$ is the restriction of $\phi_l$ to $p\mN p$.
By definition $\varphi_l$ is a $\mN$-bimodular map.
Hence, for any $x\in\mM$, $\varphi_l(pxp)=p\varphi(x)p$.
Therefore, $\phi_l$ defines a map from $p\mM p$ to $p\mM p$.
It is easy to see that $(\phi_l)_l$ is a CBAI with constant $C$ for $p\mN p\subset p \mM p$.

Proof of $(3)\Rightarrow (1)$.
Using the fact that $(1)\Rightarrow (3)$, 
it is sufficient to show that for any $n\geqslant 1$: if $\mN\subset \mM$ admits a CBAI with constant $C$, then $\mN\otimes \mM_n(\C)\subset \mM\otimes\mM_n(\C)$ admits a CBAI with constant $C$.
Let $(\varphi_l)_l$ be a CBAI with constant $C$ for $\mN\subset \mM$.
One can check that $(\varphi_l \otimes 1_{\mM_n(\C)})_l$ is a CBAI with constant $C$ for $\mN\otimes \mM_n(\C)\subset \mM\otimes\mM_n(\C)$.
\end{proof}

Let $\NM$ be a subfactor together with a tunnel-tower $M_j,\ j\in \Z$.
We recall a result due to Popa.

\begin{lemma}\cite[Proposition 2.9 and Proposition 2.10.a]{Popa_symm_env_T}\label{lemma:reduction}
Let $Q$ be an intermediate subfactor $M_i\subset Q\subset N$ where $i$ is an integer.
Then 
$$M\vee M^\op\subset \MboxM\text{ is isomorphic to }M\vee M^\op\subset M\boxtimes_{e_Q}M^\op.$$

Consider the Jones projection $e_1^\op\in M_1^\op\subset M_1\boxtimes_{e_M}M_1^\op.$
Then $$e_1^\op( M_1\vee M_1^\op) e_1^\op \subset e_1^\op( M_1\boxtimes_{e_M}M_1^\op )e_1^\op\text{ is isomorphic to }M_1\vee N^\op\subset \MboxM.$$
\end{lemma}

\begin{coro}\label{coro:Lcb}
Let $Q$ and $P$ be II$_1$ factors such that $$M_{i-1}\subset Q\subset M_i\subset M_j\subset P\subset M_{j+1}$$ for some integers $i<j$.
Then $\Lcb(Q,P)=\Lcb(N,M).$
\end{coro}

\begin{proof}
By Lemma \ref{lemma:reduction}, the compression of $ M_1\vee M_1^\op\subset M_1\boxtimes_{e_M}M_1^\op$ by $e_1^\op$ is isomorphic to $M_1\vee N^\op\subset \MboxM.$
Observe that $M\vee N^\op\subset M_1\vee N^\op$ and $M\vee N^\op\subset M\vee M^\op$ are both finite index subfactors (where everything is viewed in $\MboxM$).
Hence, by Proposition \ref{prop:CBAI_finite_index}, we have 
\begin{equation}\label{equation:NM_MM1}
\Lcb(N,M)=\Lcb(M,M_1).
\end{equation}
Consider a II$_1$ factor $\tilde Q$ such that $M_r\subset \tilde Q\subset M_{k-1}$ for some integers $r<k$.
By Lemma \ref{lemma:reduction}, the \syenin s $M_k\vee M_k^\op\subset M_k\boxtimes_{e_{M_{k-1}}}M_k^\op$ 
and $M_k\vee M_k^\op\subset M_k\boxtimes_{e_{\tilde Q} }M_k^\op$ are isomorphic.
Therefore, $\Lcb(\tilde Q,M_k)=\Lcb(M_{k-1},M_k)$.
Depending on the parity of $k$, the \stin\ of $M_{k-1}\subset M_k$ is isomorphic to $\G_{N,M}$ or to $\G_{M,M_1}$.
Therefore, $\Lcb(M_{k-1},M_k)$ is equal to $\Lcb(N,M)$ or $\Lcb(M,M_1)$ by Theorem \ref{theo:wa_for_stin}.
So by (\ref{equation:NM_MM1}) we have that $\Lcb(M_{k-1},M_k)=\Lcb(N,M)$.
Hence,
\begin{equation}\label{equation:NM_QMj}
\Lcb(N,M)=\Lcb(\tilde Q,M_k),\text{ for any $r<k$ and $\tilde Q$ such that }M_r\subset \tilde Q\subset M_{k-1}.
\end{equation}
Consider a downward basic construction $L\subset M_{i-1}\subset P$ such that $M_r\subset L$ for a certain $r$.
We have $L\subset M_{i-2}$ because $[M_{i-1}:L]=[P:M_{i-1}]>[M_k:M_{k-1}]=[M:N].$
Hence, $\Lcb(L,M_{i-1})=\Lcb(N,M)$ by (\ref{equation:NM_QMj}).
So, $\Lcb(M_{i-1},P)=\Lcb(L,M_{i-1})$ by (\ref{equation:NM_MM1}).
Therefore,
\begin{equation*}\label{equation:NM_MkP}
\Lcb(N,M)=\Lcb(M_{i-1},P).
\end{equation*}
Observe, the \syenin s $P\vee P^\op\subset P\boxtimes_{e_Q} P^\op$ and $P\vee P^\op\subset P\boxtimes_{ e_{ M_{ i - 1 } } } P^\op$ are isomorphic by Lemma \ref{lemma:reduction}.
Therefore, $$\Lcb(Q,P)=\Lcb(M_{i-1},P)=\Lcb(N,M).$$
\end{proof}

\begin{theo}\label{theo:Bisch-Haagerup}
Consider two finite groups that act outerly on a II$_1$ factor, $\sigma:H_1\loriar Aut(P)$ and $\rho:H_2\loriar Aut(P)$.
Denote by $G$ the subgroup of $Out(P)$ generated by $\sigma(H_1)$ and $\rho(H_2)$. 
Consider the Bisch-Haagerup subfactor $P^{H_1}\subset P\rtimes H_2$, where $P^{H_1}$ denotes the elements of $P$ that are fixed under the action of $H_1$. 
Then $\Lcb(P^{H_1},P\rtimes H_2)=\Lcb(G).$
\end{theo}
 
\begin{proof}
We follow the proof of \cite[Proposition 2.7]{Bisch_Popa_subfactors_PropT}.
We write $N=P^{H_1}$ and $M=P\rtimes H_2$.
Let $P_{-1}$ and $P_1$ be II$_1$ factors such that $P_{-1}\subset N\subset P$ and $P\subset M\subset P_1$ are basic constructions such that $M_{-2}\subset P_{-1}.$
By Corollary \ref{coro:Lcb}, we have $\Lcb(N,M)=\Lcb(P_{-1},P_1)$.
By \cite[Proposition 2.7]{Bisch_Popa_subfactors_PropT}, we know that $P_{-1}\subset P_1$ is isomorphic to the inclusion 
$$P_{-1}\subset P_{-1}\otimes B(\ell^2(H_1))\otimes B(\ell^2(H_2)),\ x\longmapsto \sum_{h\in H_1, k\in H_2}\rho_k\circ\sigma_h(x) f_{hh}e_{kk},$$
where $\{f_{ij} ,\ i,j\in H_1\}$ and $\{ e_{ij} , \ i , j \in H_2 \}$ are system of matrix units for $B(\ell^2(H_1))$ and $B(\ell^2(H_2))$.
Hence, $P_{-1}\subset P_1$ is isomorphic to a diagonal subfactor associated to an outer action of the group $G$.
By Corollary \ref{coro:diagonal}, we obtain $\Lcb(N,M)=\Lcb(G)$.
\end{proof}

\section{Free product}\label{sec:free_product}

Let $(\G_\al,\ \al\in \A)$ be a finite family of amenable \stin s.
In this section, we prove that the free product in the sense of Bisch and Jones of the $\G_\al$ is a weakly amenable \stin\ with constant 1 \cite{Bisch_Jones_97_Fuss_Catalan,Bisch_Jones_Free_composition}.

Bisch and Jones introduced in \cite{Bisch_Jones_97_Fuss_Catalan,Bisch_Jones_Free_composition} the free product of \stin s and \sufa s.
We recall the principal properties of this construction.
Let $\G_1$ and $\G_2$ be \stin s.
The free product of $\G_1$ and $\G_2$ is a \stin\ denoted by $\G:=\G_1*\G_2$.
There exists a chain of II$_1$ factors $N\subset P\subset M$ such that 
\begin{equation}\label{equa:free_stin}
\G_{N,P}\simeq \G_1,\ \G_{P,M}\simeq \G_2,\text{ and } \GNM\simeq \G_1*\G_2.
\end{equation}
Let $\mathcal C_1=\langle _PL^2(P)\otimes_N L^2(P)_P\rangle$ be the category of $P$-bimodules generated by the bimodule $L^2(P)\otimes_N L^2(P).$
Consider the category $\mathcal C_2=\langle _ML^2(M)_M\rangle$.
We have that $\mathcal C_1$ and $\mathcal C_2$ are in free position inside the category of all $P$-bimodules Bimod$(P)$.
In particular, if $a_1,\cdots,a_n$ are irreducible $P$-bimodules of $\mathcal C_1$ that are different from $L^2(P)$, and if 
$b_1,\cdots,b_n$ are irreducible $P$-bimodules of $\mathcal C_2$ that are different from $L^2(P)$,
then $a_1b_1\cdots a_nb_n$ is an irreducible $P$-bimodule.

We fix two \stin s $\G_1,\G_2$, and a chain of II$_1$ factors $N\subset P\subset M$ that satisfies (\ref{equa:free_stin}).
By definition, $\Lcb(\G)=\Lcb(N,M)$.
Consider a downward basic construction $$M_{-2}\subset N\subset M.$$
By Corollary \ref{coro:Lcb},  
$$\Lcb(M_{-2},P)=\Lcb(N,M).$$
Observe, the bimodule category 
$$\mathcal C=\langle _PL^2(P)\otimes_{M_{-2}}L^2(P)_P\rangle$$ is generated by $\mathcal C_1$ and $\mathcal C_2$.
 
Let $(\G_\al,\ \al\in\A)$ be a finite family of \stin s.
By iteration of the process described above, we define a free product of those \stin s $\G:=*_{\al\in\A} \G_\al.$
It satisfies the following properties:
There exists a \sufa\ $\NM$, a family of \sufa s $(N_\al\subset M_\al,\ \al\in \A)$ and a family of categories $(\mathcal C_\al,\ \al\in \A)$ such that 
\begin{enumerate}\label{jkl}
\item $\Lcb(\G)=\Lcb(N,M)$,\\
\item $\G_{N_\al,M_\al}\simeq\G_\al$,\\
\item $\mathcal C_\al$ is a subcategory of Bimod$(M)$ and is equivalent to the category of $M_\al$-bimodules $$\langle _{M_\al}L^2(M_\al)\otimes_{N_\al} L^2(M_\al)_{M_\al}\rangle,$$
\item the categories $(\mathcal C_\al,\ \al\in \A)$ are in free position inside Bimod$(M)$, and
\item the category $\mathcal C:=\langle_ML^2(M)\otimes_N L^2(M)_M\rangle$ is generated by the subcategories 
$(\mathcal C_\al,\ \al\in \A).$ 
\end{enumerate}

We will show that the \syenin\ $\TS$ associated to $\NM$ is isomorphic to an amalgamated free product over $T$.
For this purpose, we recall a construction due to Masuda that is an extension of the Longo-Rehren construction for infinite depth \sufa s of \tII\ \cite{Masuda_00_LR-construction,Longo_Rehren_95_Nets_sf}.

Consider a subfactor $\NM$.
We put $A=M\ootimes M^\op$. 
Let $\{H_k,\ k\in K\}$ be a set of representatives of the isomorphism classes of irreducible $M$-bimodules that appear in the Jones tower of $\NM$.
Denote by $B_k:=H_k\otimes \overline {H_k}^\op$ the associated $A$-bimodule and by $X$ their orthogonal direct sum $\bigoplus_{k\in K}B_k$.
For any $i,j,k\in K$, we denote by $N_{ij}^k$ the dimension of the space of intertwiners from $B_i\otimes_A B_j$ into $B_k$, 
$$\text{i.e. }N_{ij}^k:=\dim Hom_{A-A}(B_i\otimes_A B_j,B_k).$$
Let $d_k$ be the square root of the index of $M$ inside the right $M$-modular morphisms of $B(H_k)$, $d_k:=\sqrt{[End_{-M}(H_k):M]}$, and let $\{V_{ijk}^e,\ e = 1,\cdots, N_{ij}^k\}$ be a choice of an orthonormal basis of the space of intertwiners $Hom_{A-A}(B_i\otimes_A B_j,B_k)$, for any $i,j,k\in K$.
Define $$\tilde V_{ijk}:=\sum_{e=1}^{N_{ij}^k}\sqrt{ \dfrac{ d_i d_j}{d_k} } V_{ijk}^e\otimes \overline{ V_{ijk}^e     }^\op, \text{ for any }i,j,k\in K. $$
Consider a bounded vector $\xi\in B_i^\bdd$ and a vector $\eta\in B_j$.
We define $$\lambda_\xi(\eta):= \sumk \tilde V_{ijk}(\xi\otimes\eta).$$
By \cite[Lemma 2.2]{Masuda_00_LR-construction}, $\lambda_\xi$ extends to a bounded linear map from $X$ to $X$ that we continue to denote by $\lambda_\xi.$
Denote by $A(K)$ the von Neumann subalgebra of $B(X)$ generated by the set 
$$\{\lambda_\xi,\ \xi\in B_k^\bdd,\ k\in K\}.$$
By \cite[Theorem 3.4]{Masuda_00_LR-construction}, $A\subset A(K)$ is isomorphic to the \syenin\ $\TS$ associated to the subfactor $\NM$.

Let $\G=*_{\al\in \A}\G_\al$ be the free product of a finite family of extremal \stin s.
This is still an extremal \stin.
Denote by $\NM$ the \sufa\ described above that satisfies $\Lcb(N,M)=\Lcb(\G)$.
Let $K$ be the index set of the even vertices of the dual principal graph of $\NM$ and $\{ H_k,\ k\in K\}$, the associated set of irreducible $M$-bimodules.
For any $\al\in \A$, let $K_\al\subset K$ be the set of indices corresponding to the irreducible $M$-bimodules of the category $\mathcal C_\al$.
Observe, $K$ is the set of finite words with letter in the different $K_\al,\ \al\in \A$.
Let $A\subset A(K)$ be the inclusion defined by Masuda and let $A(K_\al)$ be the von Neumann subalgebra of $A(K)$ generated by the set 
$$\{ \lambda_\xi,\ \xi\in B_k^\bdd,\ k\in K_\al\}.$$
For any $\al\in \A$ there is a unique normal trace-preserving conditional expectation $E^{A(K_\al)}_A:A(K_\al)\loriar A.$
By construction and \cite[Theorem 3.4]{Masuda_00_LR-construction}, we have that $A\subset A(K_\al)$ is isomorphic to a \syenin\ associated to a subfactor of $M$ with \stin\ isomorphic to $\G_\al$.
Consider the free product of the $A(K_\al), \al\in \A$ with amalgamation over $A$ with respect to the conditional expectations $E^{A(K_\al)}_A$, see \cite[section 3.8]{Voiculescu_dykema_nica_Free_random_variables}.
We denote this \vna\ by $*_{\al\in\A} ( A(K_\al),\ E^{A(K_\al)}_A).$
There is a canonical inclusion 
$$A\subset *_{\al\in\A} ( A(K_\al),\ E^{A(K_\al)}_A).$$

\begin{prop}\label{prop:free_prod}
The two inclusions $A\subset A(K)$ and $A\subset  *_{\al\in\A} ( A(K_\al),\ E^{A(K_\al)}_A)$ are isomorphic.
\end{prop}

\begin{proof}
One can see that the family von Neumann subalgebras $(A(K_\al), \al\in\A)$ generates $A(K)$ as a \vna. 
Hence it is sufficient to show that they are free with amalgamation over $A$. 
Let $D$ be the $*$-subalgebra of $B(X)$ generated by the set $\{ \lambda_\xi, \ \xi\in B_k^\bdd,\ k\in K\}$.
Similarly, for any $\al\in\A$, we define the $*$-subalgebra $D_\al\subset B(X)$ which is generated by the set $\{ \lambda_\xi, \ \xi\in B_k^\bdd,\ k\in K_\al\}$.
The $*$-subalgebra $D_\al$ is weakly dense in $A(K_\al)$.
By \cite[Proposition 2.5.7]{Voiculescu_dykema_nica_Free_random_variables}, adapted in the amalgamated case, it is sufficient to prove that the $D_\al, \al\in\A$ are free with amalgamation over $A$.
Consider $\xi\in B_i^\bdd$ and $\eta\in B_j^\bdd$ where $i,j\in K$.
We have that 
\begin{equation}\label{equa:Masuda}
\lambda_\xi\circ\lambda_\eta=\sumk \lambda_{\xi_k}, 
\end{equation}
where $\xi_k=\tilde{V_{ijk}}(\xi\otimes \eta)$, see \cite[Proposition 2.3]{Masuda_00_LR-construction}.
Hence, for any $a\in D$ there exists a unique decomposition $a=\sumk \lambda_{\xi_k}$, such that $\xi_k\in B_k^\bdd$.
This sum is necessarily finite.
Denote by $\supp(a)$ the set of $k\in K$ such that $\lambda_{\xi_k}\neq 0$.
Consider some indices $\al_1,\cdots,\al_n\in \A$ such that $\al_1\neq \al_2\neq\cdots\neq \al_n.$
For any $i=1,\cdots,n$, consider $a_i\in D_{\al_i}\ominus A$.
Let us show that $E_A(a_1\cdots a_n)=0$, where $E_A:A(K)\loriar A$ is the unique normal trace-preserving conditional expectation.
We need to show that $\supp(a_1\cdots a_n)$ does not contain the index of the trivial bimodule $L^2(A)$.
But this is an immediate consequence of the equality (\ref{equa:Masuda}).
\end{proof}

\begin{remark}
Consider the following enlightening examples given by diagonal subfactors.
Let $\A=\{1,\cdots,m\}$ and $G_1,\cdots,G_m$ be a family of finitely generated groups.
For any $\al$ consider a diagonal subfactor $N_\al\subset M_\al$ associated to the group $G_\al$, see the construction of section \ref{sec:diagonal}.
Denote by $\mathcal G_\al$ its standard invariant and by $\G$ the free product of all the $\G_\al$.
Then the inclusion $A\subset A(K)$ considered above is isomorphic to an inclusion in a crossed product $A\subset A\rtimes G$, where $G=G_1*\cdots*G_m$.
The inclusion $A\subset A(K_\al)$ is isomorphic to an inclusion in a crossed product $A\subset A\rtimes G_\al$ for any $\al\in\A$.
\end{remark}

\begin{theo}\label{theo:free_product}
Consider a finite family of amenable \stin s $(\G_\al,\ \al\in\A)$.
Consider the free product of those standard invariants $\G:=*_{\al\in\A}\G_\al$.
Then $$\Lcb(\G)=1.$$
\end{theo}

\begin{proof}
Consider a finite family of amenable \stin s $(\G_\al,\ \al\in\A)$. 
For any $\al\in\A$, let $N_\al\subset M_\al$ be an amenable \sufa\ with \stin\ $\G_\al$.
By \cite{Bisch_Jones_97_Fuss_Catalan,Bisch_Jones_Free_composition}, and the discussion at the beginning of this section, there exists a \sufa\ $\NM$ such that 
\begin{itemize}
\item $\Lcb(N,M)=\Lcb(*_{\al\in\A}\G_\al)$,\\ 
\item its \syenin\ is isomorphic to $A\subset  *_{\al\in\A} ( A(K_\al),\ E^{A(K_\al)}_A),$ where $A\subset A(K_\al)$ is isomorphic to the \syenin\ of $N_\al\subset M_\al$ for any $\al\in\A.$
\end{itemize}
By remark \ref{rem:amenable2}, for any $\al\in\A$ there exists a sequence of normal unital completely positive $A$-bimodular maps $\varphi_{l}^\al:A(K_\al)\loriar A(K_\al)$ such that the $A$-bimodule $\varphi_{l}^\al (A(K_\al))$ is bifinite, $E^{A(K_\al)}_A\circ \varphi_{l}^\al =E^{A(K_\al)}_A$, and $\lim_l\Vert \varphi_{l}^\al(x)-x\Vert_2=0$ for any $x\in A(K_\al)$.
The restrictions of those maps to $A$ are equal to the identity by Proposition \ref{prop:scalar_functions} and the fact that they are unital.
By \cite[Theorem 4.8.5]{Brown_Ozawa_book}, there exists a normal $A$-bimodular \ucp\ map $\varphi_l:A(K)\loriar A(K)$ for any $l\gO$ that satisfies 
$$\varphi_l(a_1\cdots a_d)=\varphi_l^{\al_1}(a_1)\cdots\varphi_l^{\al_d}(a_d) \text{ for any } a_i\in A(K_{\alpha_i}) \text{ such that } \al_1\neq\al_2\neq\cdots\neq\al_d.$$
Following \cite{Ricard_Xu_06_Khintchine}, consider the Poisson kernel 
$$\mathcal T_N:=\sum_{d=0}^N (1-\frac{1}{\sqrt N})^d \mathcal P_d,$$
where $\mathcal P_d$ is the projection from $A(K)$ onto the operator space spanned by the words of length $d$.
By \cite[Proposition 3.5]{Ricard_Xu_06_Khintchine} adapted in the amalgamated case as explained in section 5 of the same article, we have that $\mathcal T_N$ is normal and completely bounded such that $\lim_{N\rightarrow\infty} \Vert \mathcal T_N\Vcb=1.$
We define the map $\phi_{N,l}:=\varphi_l\circ \mathcal T_N$.
This a normal $A$-bimodular completely bounded map such that the $A$-bimodule $\phi_{N,l}(A(K))$ is bifinite.
Further, we have $\limsup_{N,l}\Vert \phi_{N,l}\Vcb=1.$
In order to conclude, we need to show that 
\begin{equation}\label{equ:limite}
\lim_{N,l}\Vert\phi_{N,l}(x)-x\Vert_2=0 \text{ for any } x\in A(K).
\end{equation}
Let $C_{N,l}$ (resp. $c_l^\al$) be the scalar valued function associated to $\phi_{N,l}$ (resp. $\varphi_l^\al$) for any $N,l\gO$ and $\al\in\A.$
Consider an element $k\in K$.
The set $K$ is equal to the set of finite words with alternating letters in the $K_\al,\ \al\in\A$.
Hence, there exists $d\gO$, $\al_1\neq\al_2\neq\cdots\neq\al_d\in\A$, and $k_1\in K_{\al_1},\cdots,k_d\in K_{\al_d}$ such that $k=k_1\cdots k_d$.
Observe, if $N\geqslant d$ we have that 
$$C_{N,l}(k)=(1-\frac{1}{\sqrt N})^d c_l^{\al_1}(k_1)\cdots c_l^{\al_d}(k_d).$$
Since the sequence $ ( \varphi^{\al_i}_l ,\ l\gO)$ is a CBAI, by Proposition \ref{prop:scalar_functions} we have that $c_l^{\al_i}(k_i)\loriar_l1$ for any $1\leqslant i\leqslant d$.
Therefore, $C_{N,l}(k)\loriar_{N,l}1$.
Hence, by Proposition \ref{prop:scalar_functions}, we obtain (\ref{equ:limite}).
\end{proof}

\section{Tensor product}\label{sec:tensor_product}

Consider a finite family of extremal \stin s $(\G_\al=\{A^\al_{ij},\ i,j\geqslant -1\},\ \al\in\A)$.
One can form their tensor product 
$$\bigotimes_{\al\in\A}\G_\al=\{\otimes_{\al\in\A}A^\al_{ij},\ i,j\geqslant -1\},$$ 
which is still an extremal standard invariant.
This section is devoted to the proof of the following theorem:

\begin{theo}\label{theo:tensor_product}
If $(\G_\al,\ \al\in\A)$ is a finite family of (extremal) \stin s, then
$$\Lcb(\bigotimes_{\al\in\A}\G_\al)=\prod_{\al\in\A}\Lcb(\G_\al).$$
\end{theo}

The group case and the \vna\ case have been respectively proved by Cowling and Haagerup \cite{Cowling-Haagerup-89-completely-bounded-mult}, and by Sinclair and Smith \cite{Sinclair_Smith_95_Haagerup_inv}.
We adapt the proof given in \cite[Theorem 12.3.13]{Brown_Ozawa_book}.

Let $\TS$ be a \syenin\ and let $K$ be the index set of the irreducible $T$-bimodules as denoted in the previous sections.
Consider some finite subsets $F\subset E\subset K$.
Denote by 
$$\CB(S,E):=\{\sum_{k\in E}c(k)a_k,\ c(k)\in\C\}$$ 
the operator space of \cobo\ $T$-bimodular maps $\varphi:S\loriar S$ such that the range of $\varphi$ is contained in $\sum_{k\in E}L^2(\spann Tv_kT)$, see notations of Proposition \ref{prop:scalar_functions}.
Consider the free vector space $\C F$ with basis $\{\delta_k,\ k\in F\}$.
We have an inclusion of $\C F$ in the topological dual of the (finite dimensional) Banach space $\CB(S,E)$ given by 
$$u(\varphi_c)=\sum_{k\in F} u_kc(k),$$
where 
$$\varphi_c=\sum_{k\in E}c(k)a_k\in\CB(S,E)\text{ and }u=\sum_{k\in F} u_k\delta_k\in\C F.$$
We identify $\C F$ and its image in the dual space $\CB(S,E)^*$.
We write $\Vert \cdot\Vert_{\CB(S,E)^*}$ the norm of the dual space $\CB(S,E)^*$.
Here are three key observations. 
We omit the proof that are straightforward adaptation of the one given in \cite{Brown_Ozawa_book}.

\begin{lemma}\cite[Lemma 12.3.14 and 12.3.16]{Brown_Ozawa_book}\label{lemma:tensor_product}
\begin{enumerate}
\item The inclusion $\TS$ admits a CBAI of constant D if and only if for any finite subset $F\subset K$ there exists a finitely supported scalar valued function $c:K\loriar\C$ such that 
$c(k)=1$ for any $k\in F$ and $\Vert \varphi_c\Vcb\leqslant D$, where $\varphi_c=\sum_{k\in\supp(c)} c(k)a_k.$\label{lem:1}\\
\item Consider some finite subset $F\subset E\subset K$.
There exists a scalar valued function 
$$c:K\loriar\C\text{ such that }\supp(c)\subset E,\ \Vert\varphi_c\Vcb\leqslant D\text{, and }c(k)=1\text{ for any }k\in F$$
if and only if 
$$\text{for any }u=\sum_{k\in F} u_k\delta_k\in \C F,\  \vert\sum_{k\in F}u_k\vert\leqslant D\Vert u\Vert_{\CB(S,E)^*}.$$\label{lem:2}\\
\item Let $T_i\subset S_i,\ K_i,\ i=1,2$ be \syenin s together with their index set of irreducible $T_i$-bimodules. 
Consider the tensor product $\TS=T_1\ootimes T_2\subset S_1\ootimes S_2$ with the index set $K=K_1\times K_2$.
For any $u^i=\sum_{k_i\in F_i}u^i_{k_i}\in\C F_i,\ i=1,2$, we define 
$$u=u^1\times u^2=\sum_{k_1\in F_1,k_2\in F_2}u^1_{k_1}u^2_{k_2}\delta_{k_1,k_2}\in \C F_1\times F_2.$$
We have $$\Vert u\Vert_{\CB(S,E_1\times E_2)^*}\leqslant \Vert u^1\Vert_{\CB(S_1,E_1)^*}\Vert u^2\Vert_{\CB(S_2,E_2)^*}.$$\label{lem:3}
\end{enumerate}
\end{lemma}

\begin{proof}[Proof of theorem \ref{theo:tensor_product}]
By induction, it is sufficient to prove the theorem for the tensor product of two \stin s.
Let $N_i\subset M_i$ be a \sufa\ with \stin\ $\G_i$, and let $T_i\subset S_i$, $K_i$ be their \syenin\ and index set of irreducible $T_i$-bimodules for $i=1,2.$
We consider the \sufa\ $\NM$ equals to $N_1\ootimes N_2\subset M_1\ootimes M_2.$
Its \stin, \syenin, and index set of irreducible bimodules are isomorphic to $\G_1\otimes \G_2$, $T_1\ootimes T_2\subset S_1\ootimes S_2$, and $K_1\times K_2$ respectively.
We denote them by $\G,\ \TS$, and $K$.
By considering tensor product of \cobo\ maps we get that $\Lcb(\G)\leqslant \Lcb(\G_1)\Lcb(\G_2).$

Let us show that $\Lcb(\G)\geqslant \Lcb(\G_1)\Lcb(\G_2).$
Suppose $\Lcb(\G)< \Lcb(\G_1)\Lcb(\G_2).$
There exists some constants $D_i<\Lcb(\G_i),\ i=1,2$ such that $\Lcb(\G)<D_1D_2.$
By Lemma \ref{lemma:tensor_product}.\ref{lem:1} applied to $T_i\subset S_i$ and $D_i,\ i=1,2$, there exists finite subsets $F_i\subset K_i$ such that
there are no finitely supported scalar valued functions $c_i:K_i\loriar\C$, such that $c_i(k_i)=1$ for any $k_i\in F_i$ and $\Vert \varphi_{c_i}\Vcb\leqslant D_i$.
By Lemma \ref{lemma:tensor_product}.\ref{lem:1} applied to $F_1\times F_2\subset K$ and $D_1D_2$, there exists a finitely supported scalar valued function $c:K\loriar\C$ such that
$c(k)=1$ for any $k\in F_1\times F_2$, and $\Vert\varphi_c\Vcb\leqslant D_1D_2.$
The support of $c$ is finite, hence there exists finite sets $F_i\subset E_i\subset K_i,\ i=1,2$ such that $\supp(c)\subset E_1\times E_2$.
By Lemma \ref{lemma:tensor_product}.\ref{lem:2}, there exists 
$$u^i=\sum_{k_i\in F_i}u^i_{k_i}\in\C F_i,$$ 
such that 
\begin{equation}\label{equa:tensor_product}
\Vert u^i\Vert_{\CB(S_i,E_i)^*}=1\text{ and }\vert\sum_{k_i\in F_i}u^i_{k_i}\vert> D_i, \text{ for }i=1,2.
\end{equation}
Consider the element $u=u^1\times u^2=\sum_{k_1\in F_1,k_2\in F_2}u^1_{k_1}u^2_{k_2}\delta_{k_1,k_2}\in \C F_1\times F_2.$
By Lemma \ref{lemma:tensor_product}.\ref{lem:3}, we have $\Vert u\Vert_{\CB(S,E_1\times E_2)^*}\leqslant 1.$
However, $$\vert \sum_{k_1\in F_1,k_2\in F_2}u^1_{k_1}u^2_{k_2}\vert = \vert\sum_{k_1\in F_1}u^1_{k_1}\vert \vert\sum_{k_2\in F_2}u^2_{k_2}\vert>D_1D_2 \text{ by (\ref{equa:tensor_product}).}$$
This contradict Lemma \ref{lemma:tensor_product}.\ref{lem:2} applied to $u$ and $D_1D_2$.
Therefore, $\Lcb(\G)\geqslant \Lcb(\G_1)\Lcb(\G_2).$
This concludes the proof of the theorem.
\end{proof}

\begin{remark}
Using Theorem \ref{theo:free_product} and Theorem \ref{theo:tensor_product}, we can construct many hyperfinite \sufa s $R_0\subset R$ that are weakly amenable with given \cohaco.
Here is a specific construction.

Consider a finite family of exotic amenable \stin s $(\G_\al,\ \in \A)$ such as the Haagerup \stin\ \cite{Haagerup_principal_graph_small_indicies}, the Asaeda-Haagerup \stin\ \cite{Asaeda_Haagerup_exotic_subfactors} or a Temperley-Lieb-Jones \stin\ with index smaller or equal to 4.
Let $\G_1$ be their free product.
By Popa, for any $\al\in\A$ there exists a hyperfinite \sufa\ with \stin\ equal to $\G_\al$ \cite{Popa_classification_subfactors_amenable}.
By the freeness result of Vaes and the construction of Bisch and Jones, there exists a \sufa\ $N_1\subset M_1$ with \stin\ $\G_1$ where $M_1$ is the hyperfinite II$_1$ factor \cite{Vaes_09_factors_without_subfactors, Bisch_Jones_97_Fuss_Catalan,Bisch_Jones_Free_composition}.
Consider a lattice $G$ in the Lie group $Sp(n,1)$ with $n\geqslant 2$.
Consider a diagonal \sufa\ $N_2\subset M_2$ associated to an outer action of $G$ on the hyperfinite II$_1$ factor.
Then, the tensor product \sufa\ $N_1\ootimes N_2\subset M_1\ootimes M_2$ is a hyperfinite \sufa\ with \cohaco\ equal to $2n-1$ by \cite{Cowling-Haagerup-89-completely-bounded-mult}, Theorem \ref{theo:free_product} and Theorem \ref{theo:tensor_product}.
\end{remark}

\end{document}